\documentclass[11pt]{article}
\usepackage{amssymb,amsmath}
\usepackage{float}
\usepackage{color,fullpage}

\usepackage{algorithm,algorithmic}
\usepackage{amsthm}
\newcommand{\RR}{\mathbb{R}}

\newcommand{\cF}{{\mathcal{F}}}

\newtheorem{theorem}{Theorem}

\newtheorem{corollary}{Corollary}

\begin{document}

\title{On the Linear Convergence of the Cauchy Algorithm for a Class of Restricted Strongly Convex Functions}

\author{Hui Zhang\thanks{
College of Science, National University of Defense Technology,
Changsha, Hunan, 410073, P.R.China. Corresponding author. Email: \texttt{h.zhang1984@163.com}}
}



\date{\today}

\maketitle

\begin{abstract}
In this short note, we extend the linear convergence result of the Cauchy algorithm, derived recently by E. Klerk, F. Glineur, and A. Taylor, from the case of smooth strongly convex functions to the case of restricted strongly convex functions with certain form.
\end{abstract}

\newline



\section{Introduction}
The Cauchy algorithm, which was proposed by Augustin-Louis Cauchy in 1847, is also known as the gradient method with exact line search. Although this is the first taught algorithm during introductory courses on nonlinear optimization, the worst-case convergence rate question of this method was not precisely understood until very recently the authors of \cite{Klerk2016on} settled it for strongly convex, continuously differentiable functions $f$ with Lipschitz continuous gradient. This class of functions, denoted by $\cF_{\mu,L}(\RR^n)$, can be described by the following inequalities:
\begin{align}
& \|\nabla f(x)-\nabla f(y)\|\leq L\|x-y\|, \quad\forall x, y\in\RR^n; \label{Lip} \\
&\langle \nabla f(x)-\nabla f(y), x-y\rangle \geq \mu \|x-y\|^2, \quad\forall x, y\in\RR^n.\label{SC}
\end{align}
The gradient descent method with exact line search can be described as follows.
\begin{algorithm}[htb]
\caption{The gradient descent method with exact line search}  \label{alg0}
\begin{tabbing}
\textbf{Input:} $f\in \cF_{\mu,L}(\RR^n)$, $x_0\in\RR^n$ .\\

1: \textbf{for} $k=0, 1, \cdots,$ \textbf{do}\\

2: \quad $\gamma=\arg\min_{\gamma\in \RR} f(x_i-\gamma\,\nabla f(x_i))$; \quad\quad\quad ${//}$ the exact linear search\\

3: \quad $x_{i+1}=x_i-\gamma\,\nabla f(x_i)$; \quad\quad\quad\quad\quad\quad\quad\quad ${//}$  the gradient descent\\

4: \textbf{end for}
\end{tabbing}
\vspace{-10pt}
\end{algorithm}

The authors of  \cite{Klerk2016on} obtained the following result, which settles the worst-case convergence rate question of the Cauchy algorithm on $\cF_{\mu,L}(\RR^n)$.
\begin{theorem}\label{mainresult0}
 Let $f\in \cF_{\mu,L}(\RR^n)$, $x_*$ a global minimizer of $f$ on $\RR^n$, and $f_*=f(x_*)$. Each iteration of the gradient
method with exact line search satisfies
\begin{equation}\label{linearcon}
 f(x_{i+1})-f_*\leq \left( \frac{L-\mu}{L+\mu}\right)^2 (f(x_{i})-f_*), ~~i=0, 1,  \cdots.
\end{equation}
\end{theorem}
 For the case of quadratic functions in $\cF_{\mu,L}(\RR^n)$ with the form
 $$f(x)=\frac{1}{2}x^TQx+ c^Tx,$$
where $Q\in\RR^{n\times n}$ is a positive definite matrix and $c\in \RR^n$ is a vector, the result in Theorem \ref{mainresult0} is well-known. The main contribution of \cite{Klerk2016on} is extending the linear convergence of \eqref{linearcon} from quadratic case to the case of $\cF_{\mu,L}(\RR^n)$. However, Theorem \ref{mainresult0} can not completely answer what will happen when the matrix $Q$ is positive semidefinite. In this note, we will further extend the linear convergence of \eqref{linearcon} to a certain class of restricted strongly convex (RSC) functions, or more concretely to the functions that have the form of $f(x)=h(Ax)$, where $h\in \cF_{\mu,L}(\RR^m)$ and $A\in \RR^{m\times n}$ with $m\leq n$, which cover the quadratic function $f(x)=\frac{1}{2}x^TQx+ c^Tx$ with positive semidefinite matrix $Q$.

 The concept of the restricted strongly convex was proposed in our previous paper \cite{zhang2015restricted,zhang2013gradient} and is strictly weaker than the strong convex; for detail and its recent development we refer the reader to \cite{zhang2015The,Frank2015linear,zhang2016new}. It is not difficulty to see that $f(x)=h(Ax)$ is generally not strongly convex even $h(\cdot)$ is unless $A$ has full column-rank. But it is always restricted strongly convex since it satisfies the following inequality that was proposed to define the RSC function:
 \begin{equation}
 \langle \nabla f(x), x-x^\prime\rangle \geq \nu \|x-x^\prime\|^2, \quad\forall x, y\in\RR^n,\label{RSC}
 \end{equation}
 where $x^\prime$ is the projection point of $x$ onto the minimizer set of $f(x)$, and $\nu$ is a positive constant.
 A decisive difference to strong convexity is that the set of minimizers of a RSC function need not be a singleton.

 Now, we state our main result as follows:
 \begin{theorem}\label{mainresult}
 Let $h\in \cF_{\mu,L}(\RR^m)$, $A\in \RR^{m\times n}$ with $m\leq n$ and having full row-rank, $f(x)=h(Ax)$, $x_*$ belong the minimizer set of $f$ on $\RR^n$, and $f_*=f(x_*)$. Denote $\kappa_h=\frac{\mu}{L}$ and $\kappa(A)=\frac{\lambda_{\min}(AA^T)}{\lambda_{\max}(AA^T)}$, where $\lambda_{\min}(AA^T)$ and $\lambda_{\max}(AA^T)$ stand for the smallest and largest eigenvalues of $AA^T$, respectively. Each iteration of the gradient method with exact line search satisfies
\begin{equation}\label{mainconv}
 f(x_{i+1})-f_*\leq \left(\frac{2-\kappa(A)-\kappa(A)\kappa_h}{2-\kappa(A)+\kappa(A)\kappa_h}\right)^2 (f(x_{i})-f_*), ~~i=0, 1,  \cdots.
\end{equation}
\end{theorem}
When $A$ is the identity matrix, Theorem \ref{mainresult} exactly recovers Theorem \ref{mainresult0} since $\kappa(A)=1$. From this sense, our main result is a further extension of Theorem \ref{mainresult0}. Recently, there appeared several papers \cite{zhang2015restricted,Frank2015linear,I2015Linear,zhang2016new}, the authors of which derived linear convergence results of gradient method with \textsl{fixed} step-length for RSC functions. The author of \cite{Frank2015linear} analyzed the linear convergence rate of gradient method with exact linear search for RSC functions, but he only showed the existence of linear convergence rate. Our novelty here lies in the exact expression \eqref{mainconv} for linear convergence.

For generally quadratic case, we have the following consequence.
\begin{corollary}\label{corr}
Let $f(x)=\frac{1}{2}x^TQx+ c^Tx,$
where $Q\in\RR^{n\times n}$ is a positive semidefinite matrix and $c\in \RR^n$ is a vector; both of them are given. Assume that there is a vector $x_*$ minimizing $f(x)$ on $\RR^n$ and denote $f_*=f(x_*)$. Then, each iteration of the gradient method with exact line search satisfies
\begin{equation}\label{lincon}
 f(x_{i+1})-f_*\leq \left(1- \frac{\lambda^{++}_{\min}(Q)}{ \lambda_{\max}(Q)}\right)^2 (f(x_{i})-f_*), ~~i=0, 1,  \cdots,
\end{equation}
where $\lambda^{++}_{\min}(Q)$ stands for the smallest strictly positive eigenvalue of $Q$ and $\lambda_{\max}(Q)$ is the largest eigenvalue of $Q$.
\end{corollary}
 When $Q$ is positive definite, applying Theorem \ref{mainresult0} to the quadratic function $f(x)=\frac{1}{2}x^TQx+ c^Tx$, we can get  $$f(x_{i+1})-f_*\leq \left(\frac{\lambda_{\max}(Q)-\lambda_{\min}(Q)}{ \lambda_{\max}(Q)+\lambda_{\min}(Q)}\right)^2(f(x_{i})-f_*), ~~i=0, 1,  \cdots.$$
But from Corollary \ref{corr}, we can only obtain a worse linear convergence rate $\left(1- \frac{\lambda_{\min}(Q)}{ \lambda_{\max}(Q)}\right)^2$. Therefore, we wonder whether one can improve the rate  in Corollary \ref{corr} from $\left(1- \frac{\lambda^{++}_{\min}(Q)}{ \lambda_{\max}(Q)}\right)^2$ to $$\left( \frac{\lambda_{\max}(Q)-\lambda^{++}_{\min}(Q)}{ \lambda_{\max}(Q)+\lambda^{++}_{\min}(Q)}\right)^2.$$

\section{Proof}
\begin{proof}[Proof of Theorem \ref{mainresult}] We divide the proof into three steps.

\textbf{Step 1.} Denote $\lambda=\sqrt{\lambda_{\max}(AA^T)}$ and let $\widetilde{h}_\lambda(y)=h(\lambda y)$. Since $\nabla \widetilde{h}_\lambda(y)=\lambda\cdot \nabla h(\lambda y)$, from the definition of $h\in \cF_{\mu,L}(\RR^m)$, we have that
\begin{align}
& \|\nabla \widetilde{h}_\lambda(y)-\nabla \widetilde{h}_\lambda(z)\|\leq \lambda^2 L\|y-z\|, \quad\forall y, z\in\RR^m;   \\
&\langle \nabla \widetilde{h}_\lambda(y)-\nabla \widetilde{h}_\lambda(z), y-z\rangle \geq \lambda^2 \mu \|y-z\|^2, \quad\forall y, z\in\RR^m.
\end{align}
Denote $\widetilde{u}=\lambda^2 \mu. \widetilde{L}=\lambda^2 L$ and $\widetilde{A}=\lambda^{-1}A$; Then we can conclude that $\widetilde{h}_\lambda(y)\in \cF_{\widetilde{\mu},\widetilde{L}}(\RR^m)$ and
$$f(x)=h(Ax)=\widetilde{h}_\lambda(\widetilde{A}x).$$

\textbf{Step 2.} The iterates, generated by the gradient method with exact line search, satisfy the following two conditions for $i=0, 1, \cdots,$
\begin{align}
& x_{i+1}-x_i+\gamma \nabla f(x_i)= 0,~~ \textrm{for~some}~~ \gamma >0;   \\
& \nabla f(x_{i+1})^T(x_{i+1}-x_i)= 0,\label{cond1}
\end{align}
where the first condition follows from the gradient descent step, and the second condition is an alternative expression of the exact linear search step. Since $\gamma>0$, it holds that successive gradients are orthogonal, i.e.,
\begin{equation}\label{cond2}
 \nabla f(x_{i+1})^T\nabla f(x_i)= 0, ~i=0, 1, \cdots.
\end{equation}
Note that $\nabla f(x)=\widetilde{A}^T\nabla \widetilde{h}_\lambda(\widetilde{A}x)$. Denote $y_i=\widetilde{A}x_i, i=0, 1, \cdots$; Then, in light of the conditions \eqref{cond1} and \eqref{cond2}, we can get that
\begin{align}
& \nabla \widetilde{h}_\lambda(y_{i+1})^T(y_{i+1}-y_i)= 0, ~i=0, 1, \cdots; \label{cond3}  \\
& \nabla \widetilde{h}_\lambda(y_{i+1})^T\widetilde{A}\widetilde{A}^T\nabla \widetilde{h}_\lambda(y_i)= 0, ~i=0, 1, \cdots. \label{cond4}
\end{align}

\textbf{Step 3.} Following the arguments in \cite{Klerk2016on}, we consider only the first iterate, given by $x_0$ and $x_1$, as well as the minimizer $y_*$ of $\widetilde{h}_\lambda(y)\in \cF_{\widetilde{\mu},\widetilde{L}}(\RR^m)$.

Set $h_i=\widetilde{h}_\lambda(y_i)$ and $g_i=\nabla \widetilde{h}_\lambda(y_i)$ for $i\in\{*, 0, 1\}$. Let $\epsilon =1-\kappa(A)$. Then, the following five inequalities are satisfied:
\begin{enumerate}
  \item[1:]~~ $h_0\geq h_1+g^T_1(y_0-y_1)+\frac{1}{2(1-\widetilde{\mu}/\widetilde{L})}\left(\frac{1}{\widetilde{L}}\|g_0-g_1\|^2+\widetilde{\mu} \|y_0-y_1\|^2-2\frac{\widetilde{\mu}}{\widetilde{L}}(g_1-g_0)^T(y_1-y_0)\right)$
  \item[2:]~~ $h_*\geq h_0+g^T_0(y_*-y_0)+\frac{1}{2(1-\widetilde{\mu}/\widetilde{L})}\left(\frac{1}{\widetilde{L}}\|g_*-g_0\|^2+\widetilde{\mu} \|y_*-y_0\|^2-2\frac{\widetilde{\mu}}{\widetilde{L}}(g_0-g_*)^T(y_0-y_*\right)$
  \item[3:]~~ $h_*\geq h_1+g^T_1(y_*-y_1)+\frac{1}{2(1-\widetilde{\mu}/\widetilde{L})}\left(\frac{1}{\widetilde{L}}\|g_*-g_1\|^2+\widetilde{\mu} \|y_*-y_1\|^2-2\frac{\widetilde{\mu}}{\widetilde{L}}(g_1-g_*)^T(y_1-y_*)\right)$
  \item[4:]~~ $0\geq g^T_1(y_1-y_0)$
  \item[5:]~~ $0\geq g^T_0g_1-\epsilon\|g_0\|\|g_1\|,$
\end{enumerate}
where the first three inequalities are the $\cF_{\mu,L}$-interpolability conditions (see Theorem 4 in \cite{Taylor2016smooth}), the fourth inequality is a relaxation of \eqref{cond3}. It remains to show the fifth inequality. Indeed, in terms of \eqref{cond4}, we have $g^T_o\widetilde{A}\widetilde{A}^Tg_1=0$ and hence by the Cauchy-Schwartz inequality we can derive that
$$g^T_0g_1=g^T_o(I-\widetilde{A}\widetilde{A}^T)g_1\leq \|I-\widetilde{A}\widetilde{A}^T\|\cdot\|g_0\|\|g_1\|,$$
where $I$ is the identity matrix.
Since $\widetilde{A}=\lambda^{-1}A=\frac{A}{\sqrt{\lambda_{\max}(AA^T)}}$, it holds that
$$\|\widetilde{A}\widetilde{A}^T\|=\frac{\|AA^T\|}{\lambda_{\max}(AA^T)}\leq 1$$ and hence
$$\|I-\widetilde{A}\widetilde{A}^T\|=1-\lambda_{\min}(\widetilde{A}\widetilde{A}^T)
=1-\frac{\lambda_{\min}(AA^T)}{\lambda_{\max}(AA^T)}=1-\kappa(A).$$
Therefore,
$$g^T_0g_1 \leq (1-\kappa(A))\|g_0\|\|g_1\|=\epsilon \|g_0\|\|g_1\|.$$
Since $A$ is full row-rank, it is not difficulty to see that $0\leq \epsilon <1$. Now, we can repeat the arguments in the proof Theorem 5.1 in \cite{Klerk2016on} and get
\begin{equation}
 h_1-h_*\leq \rho_\epsilon^2(h_1-h_*),
\end{equation}
where $\rho_\epsilon=\frac{1-\kappa_\epsilon}{1+\kappa_\epsilon}$ and $\kappa_\epsilon=\frac{\widetilde{\mu}(1-\epsilon)}{\widetilde{L}(1+\epsilon)}$. After a simple calculus, we obtain
$$\rho_\epsilon=\frac{2-\kappa(A)-\kappa(A)\kappa_h}{2-\kappa(A)+\kappa(A)\kappa_h}.$$
Finally, noting that $f_*=h_*$ and for $i\in\{0,1\}$ it holds
$$f(x_i)=h(Ax_i)=\widetilde{h}_\lambda(\widetilde{A}x_i)=\widetilde{h}_\lambda(y_i)=h_i,$$
we get
$$f(x_1)-f_*\leq \left(\frac{2-\kappa(A)-\kappa(A)\kappa_h}{2-\kappa(A)+\kappa(A)\kappa_h}\right)^2(f(x_0)-f_*),$$
from which the conclusion follows. This completes the proof.
\end{proof}

\begin{proof}[Proof of Corollary \ref{corr}]
Since $\nabla f(x)=Qx+c$, by the assumption we have that $-Qx_*=c$. Let $m=\textrm{rank}(Q)$ and let $Q=U\Sigma U^T$ be the reduced singular value decomposition, where $U\in\RR^{n\times m}$ has orthogonal columns, and $\Sigma\in\RR^{m\times m}$ is a diagonal matrix with the nonzero eigenvalues of $Q$ as its diagonal entries. Denote $B=U\Sigma^{\frac{1}{2}}$, and $u=-\Sigma^{\frac{1}{2}}U^Tx_*$; Then
$$f(x)=\frac{1}{2}x^TBB^Tx+(B^Tx)^Tu=\frac{1}{2}\|B^Tx+u\|^2-\frac{1}{2}\|u\|^2.$$
Let $h(y)=\frac{1}{2}\|y+u\|^2-\frac{1}{2}\|u\|^2$. Then we have $f(x)=h(B^Tx)$ with $h(y)\in \cF_{1,1}(\RR^m)$ and $B^T$ having full row-rank. Note that $\kappa_h=1$ and $$\kappa(B^T)=\frac{\lambda_{\min}(B^TB)}{\lambda_{\max}(B^TB)}=\frac{\lambda^{++}_{\min}(Q)}{ \lambda_{\max}(Q)}.$$
Therefore, the desired result follows from Theorem \ref{mainresult}. This completes the proof.
\end{proof}

\section*{Acknowledgements}
This work is supported by the National Science Foundation of China (No.11501569 and No.61571008).

\bibliographystyle{abbrv}

\end{document}